\documentclass[preprint,12pt]{elsarticle}
\usepackage{latexsym}
\usepackage{amssymb}
\usepackage{amsmath}
\usepackage{amsthm}
\usepackage{verbatim}

\numberwithin{equation}{section}

\newtheorem{theorem}{Theorem}[section]
\newtheorem{lemma}{Lemma}[section]
\newtheorem{proposition}{Proposition}[section]
\newtheorem{remark}{Remark}[section]
\newtheorem{definition}{Definition}[section]

\journal{}

\begin{document}

\begin{frontmatter}



\title{Existence of multiple periodic solutions to a semilinear  wave equation with $x$-dependent coefficients
%\tnoteref{ack}
}
%\tnotetext[ack]{This work is partially supported by NSFC Grants (nos. 11322105 and 11671071).}


\author[1]{Hui Wei}
\ead{weihui01@163.com}
\author[1,2]{Shuguan Ji\corref{cor}}
\ead{jishuguan@hotmail.com}
\address[1]{School of Mathematics and Statistics and Center for Mathematics and Interdisciplinary Sciences, Northeast Normal University, Changchun 130024, P.R. China}
\address[2]{School of Mathematics, Jilin University, Changchun 130012, P.R. China}
\cortext[cor]{Corresponding author.}

\begin{abstract}
This paper is concerned with the periodic (in time) solutions to an one-dimensional semilinear wave equation with $x$-dependent coefficient. Such a model arises from the forced vibrations
of a nonhomogeneous string and propagation of seismic waves in nonisotropic media. By combining variational methods with saddle point reduction technique, we obtain the existence of at
least three periodic solutions whenever the period is a rational multiple of the length of the spatial interval. Our method is based on a delicate analysis for the asymptotic character of
the spectrum of the wave operator with $x$-dependent coefficients, and the spectral properties play an essential role in the proof.
\end{abstract}


\end{frontmatter}

% \linenumbers

\section{Introduction}

In this paper, we consider the existence of multiple periodic solutions to the semilinear wave equation with $x$-dependent coefficients
\begin{equation}
\label{eqa:1-1}
\rho(x) u_{tt} -  (\rho(x)u_x)_x = a\rho(x) u + f(t, x, u), \,  t\in \mathbb{R},\, 0<x<\pi,
\end{equation}
with the Dirichlet boundary conditions
\begin{equation}
\label{eqa:1-2}
u(t, 0) = u(t, \pi) = 0, \, t\in \mathbb{R},
\end{equation}
and the periodic conditions
\begin{equation}
\label{eqa:1-3}
u(t+T, x) = u(t, x), \ u_t(t+T, x) = u_t(t, x),  \, t\in \mathbb{R}, \, 0<x<\pi,
\end{equation}
where $a>0$  is a constant and $f$ is a given $T$-periodic function in time $t$.

Equation (\ref{eqa:1-1}) originates from the forced vibrations of a bounded nonhomogeneous string and the propagation of seismic waves in nonisotropic media
(see e.g. \cite{Barbu.(1996), Barbu.(1997)a, Barbu.(1997)b, Ji.(2008), Ji.(2009), Ji.(2006), Ji.(2007), Ji.(2011), Ru04, Ru17} and references therein).
More precisely, the vertical displacement $u(t, z)$ at time $t$ and depth $z$ of a plane seismic wave is described by the equation
\begin{equation}
\label{eqa:1-4}
\mu(z) u_{tt} -  (\nu(z)u_z)_z = 0
\end{equation}
with some initial conditions in $t$ and boundary conditions in $z$, where $\mu(z)$ is the rock density and $\nu(z)$ is the elasticity coefficient.
By the change of variable
$$x =\int_0^z \left(\frac{\mu(s)}{\nu(s)}\right)^{1/2} \textrm{d}s,$$
equation (\ref{eqa:1-4}) is transformed into
$$\rho(x) u_{tt} -  (\rho(x)u_x)_x=0,$$
where $\rho=(\mu\nu)^{1/2}$ is called the acoustic impedance function.

It is well known that the case of $\rho\equiv C\neq 0$ (a nonzero constant) corresponds to the classical wave equation, which is called the one with constant coefficients for
distinguishing it from the one with $x$-dependent coefficients discussed here. The problem of finding periodic solutions of nonlinear wave equation with constant coefficients
has received a great deal of attention since the original work \cite{R.(1967)} of Rabinowitz. By using the variational methods, he obtained the existence of periodic solutions
for the weakly nonlinear homogeneous string whenever the period $T$ is a rational multiple of the length of the spatial interval. Thereafter, many authors, such as  Br\'{e}zis,
Chang, Nirenberg etc., have used and developed the variational methods, topological degree and index theory to obtain a lot of results on the existence and multiplicity of periodic solutions for
the problem with various nonlinearities (see e.g. \cite{Amann.(1979), B.(1983), B.(1978), Chang.(1981), C.1982, Chen.(2014), Chen.(2016), Chen.(2017), Craig.(1993), I.2000, R.(1978), R.(1984), T.(2006)} and the
references therein).

%In particular,  Chang \cite{C.1982}, Izydorek \cite{I.2000} and Tanaka \cite{T.(2006)} considered the one-dimensional asymptotically linear wave equation $u_{tt}-
%u_{xx} + f(t, x, u)=0$ and obtained the existence of at least three periodic solutions. Recently, Chen and Zhang \cite{Chen.(2017)} studied the existence of multiple periodic solutions for the wave equation $u_{tt} - \Delta u = %f(t, x, u)$ in an $N$-dimensional ball with radius $R$, where $f(t, x, u)$ is asymptotically linear in $u$ at both $0$ and $\infty$, and monotone. They acquired at least three radially symmetric solutions with time period $T$ %when $N=2$ or $N>3$ is odd, and $R/T = k/4, k \in \mathbb{Z^+}$. The proof is essentially based on the spectral properties of the radially symmetric wave operator.

On the other hand,  the problem of finding periodic solutions for the nonlinear wave equation with $x$-dependent coefficients was studied by Barbu and Pavel  in \cite{Barbu.(1996), Barbu.(1997)a, Barbu.(1997)b} for the first time. In \cite{Barbu.(1997)a}, Barbu and Pavel considered the existence and regularity of periodic solutions for such wave equation with sublinear nonlinearity under the Dirichlet boundary conditions. For the case the nonlinear
term having power-law growth, Rudakov \cite{Ru04} proved the existence of periodic solutions under the Dirichlet boundary conditions. Later, Ji and his collaborators obtained some related results for the general Sturm-Liouville
boundary value problem \cite{Ji.(2008), Ji.(2006)}, and periodic and anti-periodic boundary value problem \cite{Ji.(2009), Ji.(2007)}.
In \cite{W.2009}, by using topological degree methods, Wang and An obtained an existence result on periodic solution of the problem with resonance and the sublinear nonlinearity. Afterwards, Ji and Li \cite{Ji.(2011)} obtained an existence result of periodic solution for $\eta_{\rho}(x)=0$ under the Dirichlet boundary conditions, which actually solves an open problem posted in \cite{Barbu.(1997)a}. Recently,  Ji et al. \cite{Ji.(2016)} obtain the existence and multiplicity of periodic solutions for the Dirichlet-Neumann boundary value problem of a wave equation with $x$-dependent coefficients
by using the Leray-Schauder degree theory. The restriction to such type of boundary value problem is essentially due to the possible loss of the compactness of the inverse operator on its range.

In this paper, we intend to pay close attention to the existence of multiple periodic solutions of wave equation with $x$-dependent coefficients. By combining variational methods with saddle point reduction technique, we obtain the existence of at least three periodic solutions whenever the period is a rational multiple of the length of the spatial interval. Our method is based on a delicate analysis for the asymptotic character of
the spectrum of the wave operator with $x$-dependent coefficients, and the spectral properties play an essential role in the proof.

Denote $\tilde{f}(t, x, u) = \frac{f(t, x, u)}{\rho(x)}$. Throughout this paper, we assume  $T$ is a rational multiple of $\pi$ which can be rewritten as
 $$T = 2\pi \frac{p}{q}$$
for some relatively prime positive integers $p$ and $q$. Moreover, we make the following assumptions:\\
(A1) $\rho(x) \in H^2(0, \pi)$ satisfies $\rho(x)>0$ for $x\in[0,\pi]$,  and
$$\rho_0 = \textrm{ess}\inf \eta_{\rho}(x) >0,$$
where
$$\eta_{\rho}(x) = \frac{1}{2} \frac{\rho''}{\rho}-\frac{1}{4} \left(\frac{\rho'}{\rho}\right)^2.$$
(A2) $\tilde{f}(t, x, u) \in C^1(\mathbb{R}\times (0, \pi) \times\mathbb{R})$,  $\tilde{f}(t+T, x, u) = \tilde{f}(t, x, u)$, and
\begin{eqnarray}
\tilde{f}(t, x, u)  = o(|u|), \ \ \textrm{as} \ u\rightarrow 0  \ \ \textrm{uniformly in} \ (t, x),
\label{eqa:1.2}
\end{eqnarray}
and $\tilde{f}(t, x, u)$ is asymptotically linear in $u$ at $\infty$ in the following sense: there exists a constant $b >0 $ such that
\begin{eqnarray}
\tilde{f}(t, x, u) - b u = o(|u|), \ \ \textrm{as} \ |u|\rightarrow \infty  \ \ \textrm{uniformly in} \ (t, x).
\label{eqa:1.3}
\end{eqnarray}

The rest of this paper is organized as follows. In Sect. \ref{sec:2}, we give some preliminaries and state the main result.
In Sect. \ref{sec:3}, we first characterize the solutions of problem \eqref{eqa:1-1}--\eqref{eqa:1-3} as the critical points of the corresponding variational problem. Then, with the aid of the saddle point reduction technique, we reduce the critical point of the variation problem from an infinite dimensional space to a finite dimensional subspace. In Sect. \ref{sec:4}, we prove reduction functional satisfies the $(PS)_c$ condition for any $c \in \mathbb{R}$. In
Sect. \ref{sec:5} and Sect. \ref{sec:6}, we devote to the proof of the bounds of reduction functional and the main result respectively.

\section{Preliminaries and main result}

\setcounter{equation}{0}
\label{sec:2}

Set $\Omega = (0, T) \times (0, \pi)$, and denote
\begin{equation*}
\Psi = \{\psi \in C^\infty(\Omega) : \psi(t,0) = \psi(t,\pi) = 0, \psi(0,x) = \psi(T,x), \psi_t(0,x) = \psi_t(T,x)\},
\end{equation*}
and
$$L^r(\Omega) = \Big\{ u: \|u\|^r_{L^r(\Omega)} = \int_\Omega |u(t,x)|^r \rho(x) \textrm{d}t \textrm{d}x <\infty\Big\}, \, \, r\geq 1.$$
It is well known that $\Psi$ is dense in $L^r(\Omega)$ for any $r\geq 1$, and $L^2(\Omega)$ is a Hilbert space with the inner product
$$\langle u, v \rangle = \int_\Omega u(t,x)  \overline{v(t,x)} \rho(x)\textrm{d}t \textrm{d}x, \ \forall u, v \in L^2(\Omega).$$

We rewrite (\ref{eqa:1-1})--(\ref{eqa:1-3}) on $\Omega$ in the
 following form
 \begin{eqnarray}
\label{eqa:2-1}
 && \rho(x)u_{tt}-(\rho(x)u_{x})_{x}=a\rho(x) u +f(t,x,u), \ \
 (t,x)\in \Omega,\\
\label{eqa:2-2}
&&u(t,0)=u(t,\pi)=0, \ \  t\in(0,T), \\
\label{eqa:2-3}
 &&u(0,x)=u(T,x),\ \ u_{t}(0,x)=u_{t}(T,x),\ \
x\in(0,\pi).
 \end{eqnarray}

\begin{definition}
A function $u \in L^r(\Omega)$ is called a weak solution of problem \eqref{eqa:2-1}--\eqref{eqa:2-3} if it satisfies
$$\int_\Omega u(\rho\psi_{tt} - (\rho\psi_x)_x){\rm d}t {\rm d}x - \int_\Omega (a u + \tilde{f}(t, x, u))\psi \rho{\rm d}t {\rm d}x = 0, \ \  \forall \psi \in \Psi. $$
\end{definition}

Define the linear operator $L_0$ by
$$L_0 \psi = \rho^{-1}\left(\rho \psi_{tt} - (\rho \psi_x)_x\right), \ \forall \psi \in \Psi,$$
and denote its extension in $L^2(\Omega)$ by $L$. It is known that $L$ is a selfadjoint operator (see \cite{Barbu.(1997)a}), and $u\in L^2(\Omega)$ is a weak solution of problem (\ref{eqa:2-1})--(\ref{eqa:2-3}) if and only if $Lu=au+\tilde{f}$.

For the study of periodic solutions of problem (\ref{eqa:2-1})--(\ref{eqa:2-3}), we need to use the following
complete orthonormal system of eigenfunctions $\{\phi_m(t)\varphi_n(x) : m\in \mathbb{Z}, n\in \mathbb{N}\}$ in $L^2(\Omega)$ (see \cite{Y.(1980)}), where
\begin{equation*} \label{eqa:2.1}
\phi_m(t) = T^{-\frac{1}{2}} e^{i\mu_m t}, \ \mu_m = 2m\pi T^{-1}, \ m \in \mathbb{Z},
\end{equation*}
and $\lambda_n$, $\varphi_n(x)$ are given by the Sturm-Liouville problem
\begin{equation}\label{eqa:2.2}
-(\rho(x)\varphi'_n(x))' = \lambda^2_n \rho(x) \varphi_n(x), \ \varphi_n(0)=\varphi_n(\pi)=0, \ n\in \mathbb{N}.
\end{equation}

\begin{lemma}[\cite{Barbu.(1997)a}] \label{lem:2.1}
Assume that $\rho(x)$ satisfies (A1), then
the eigenvalues of problem \eqref{eqa:2.2} have the form
$$\lambda_n = n +\theta_n \ with \ \theta_n \rightarrow 0 \ as \ n \rightarrow \infty,$$
where
\begin{equation}\label{eqa:2.3}
0<\frac{\rho_2}{n} \leq \sqrt{n^2 + \rho_0}-n \leq \theta_n \leq \sqrt{n^2 + \rho_1}-n \leq \frac{\rho_1}{2n},
\end{equation}
and $\rho_1 = \frac{2}{\pi} \int^{\pi}_0 \eta_{\rho}(x) \textrm{d}x$, $\rho_2 = \sqrt{\rho_0 +1} -1$.
\end{lemma}

By Lemma \ref{lem:2.1}, the eigenvalues of operator $L$ can be rewritten as
$$\lambda_{nm} = \lambda_n^2-\mu_m^2= p^{-2}(np + \theta_np-mq)(np + \theta_np+mq).$$
Thus,  when $np\neq |m|q$, it is easy to verify that $|\lambda_{nm}| \rightarrow \infty$ as $m, n\rightarrow  \infty$.  On the other hand, when  $np = |m|q$, by \eqref{eqa:2.3} we have
\begin{equation*}
2\rho_2\leftarrow(\frac{\rho_2}{n})^2 + 2\rho_2 \leq \lambda_{nm}= \theta_n (2n+\theta_n) \leq \rho_1+(\frac{\rho_1}{2n})^2\rightarrow \rho_1,
\end{equation*}
as $n \rightarrow \infty$.

Denote the set of eigenvalues of operator $L$ by
$$\Lambda(L)=\{\lambda_{nm}: \lambda_{nm}=\lambda_n^2-\mu_m^2\}.$$
The above statement shows that $\Lambda(L)$  has at least one accumulation point in $[2\rho_2, \rho_1]$. Therefore we have the following lemma.
\begin{lemma}\label{lem:2.a}
Let assumption $(A1)$ hold. Then\\
(i) $L$ has at least one essential spectral point, and all of them belong to $[2\rho_2, \rho_1]${\rm ;}\\
(ii) If $\lambda \in \Lambda(L)$ and $\lambda \notin [2\rho_2, \rho_1]$, then $\lambda$ is isolated and its multiplicity is finite.
\end{lemma}

If $(a, b) \cap \Lambda(L)\neq \emptyset$ and $b$  satisfies $\rho_1 < a < b$, by Lemma \ref{lem:2.a}, we can define
$$b^- = \max\{ \lambda \in \Lambda(L) :\lambda < b\}, \ \ b^+ = \min\{ \lambda \in \Lambda(L) :\lambda > b\}.$$
It is obvious that $b^- < b < b^+$.

Denote $e = b^+ -a$. The main result of this paper is as follows.
\begin{theorem}\label{th:2.1}
  Assume $a$, $b \notin \Lambda(L)$, $\rho_1<a<b^+ -b$, $(a, b) \cap \Lambda(L)\neq \emptyset$, and (A1)--(A2) hold. If $\tilde{f}$ is increasing in $u$ and there exists a constant $\kappa > 0$ such that
$$\frac{\partial \tilde{f}}{\partial u}(t, x, u) \leq e - \kappa,  \quad \forall (t, x, u) \in \mathbb{R}\times (0, \pi) \times\mathbb{R}.$$
Then the problem \eqref{eqa:2-1}--\eqref{eqa:2-3} has at least three T-periodic solutions.
\label{the:2.2}
\end{theorem}

\section{Variational problem and its reduction}
\label{sec:3}

 In what follows, we always assume $a$, $b$ satisfy the conditions in Theorem \ref{th:2.1}. Since $a \notin \Lambda(L)$ and $a >\rho_1$, then  there exists a constant $\delta > 0$ such that
\begin{eqnarray}\label{eqa:2.4}
|\lambda_{nm} - a|\geq \delta >0, \ m \in \mathbb{Z}, n\in \mathbb{N}.
\end{eqnarray}

We define the working space
$$E= \Big\{u \in L^2(\Omega):  \|u\|^2_E = \sum\limits_{n,m}|\lambda_{nm} - a| |\alpha_{nm}|^2 < \infty\Big\},$$
where $\alpha_{nm}$ denote the Fourier coefficients of $u\in L^2(\Omega)$, i.e.,
$$u=\sum_{n, m}\alpha_{nm}\varphi_n(x)\phi_m(t), \ \ \alpha_{nm}=\int_{\Omega}u\varphi_n\overline{\phi}_m\rho{\rm d}x{\rm d}t.$$
The estimate \eqref{eqa:2.4} indicates  $\|\cdot\|_E$ is a norm. Furthermore, $E$ is a Hilbert space equipped with the inner product $\langle u, v\rangle_0 = \sum\limits_{n, m}|\lambda_{nm} - a| \alpha_{nm} \overline{\beta}_{nm}$, where $\alpha_{nm}$ and $\beta_{nm}$ are the Fourier coefficients of $u$ and $v$ respectively.

From \eqref{eqa:2.4}, we have
\begin{eqnarray}\label{eqa:2.7}
\|u\|^2_{L^2(\Omega)} = \sum\limits_{n,m} |\alpha_{nm}|^2 \leq \delta^{-1} \sum\limits_{n,m}|\lambda_{nm} - a| |\alpha_{nm}|^2 = \delta^{-1} \|u\|^2_E,
\end{eqnarray}
which implies that the continuous embedding $E \hookrightarrow L^2(\Omega)$. Moreover, for $1 \leq r \leq 2$, the continuous embedding $L^2(\Omega) \hookrightarrow L^r(\Omega)$ implies that there exists a constant $C=C(r)$ such that
\begin{eqnarray}\label{eqa:2.8}
\|u\|_{L^r(\Omega)} \leq C \|u\|_E, \ \ 1\leq r\leq 2.
\end{eqnarray}

Since $a \in (\rho_1, b^+ -b)$, then
\begin{eqnarray}\label{eqa:2.5}
e = b^+ -a>b.
\end{eqnarray}

If $\lambda_{nm} > b$, we have
\begin{eqnarray}\label{eqa:2.6}
\lambda_{nm}-a\geq e>b,  \ m \in \mathbb{Z}, n\in \mathbb{N}.
\end{eqnarray}

Define the energy functional
\begin{eqnarray}\label{eqa:2.9}
\Phi(u)  = \frac{1}{2}\langle(L-a)u, u\rangle - \int_\Omega \widetilde{F}(t,x, u) \rho\textrm{d}t \textrm{d}x, \ \ \forall u  \in E,
\end{eqnarray}
where $\widetilde{F}(t,x, u) = \int^u_0 \tilde{f}(t,x, s) \textrm{d}s$. In addition, by \eqref{eqa:1.2} and the assumption on $\tilde{f}$ in Theorem \ref{th:2.1} , it is easy to see $\widetilde{F}(t,x, u) \geq 0$ for any $u \in E$.

Consequently, $\Phi$ is a $C^1$ functional on $E$ and
\begin{eqnarray}\label{eqa:2.10}
\langle \Phi'(u), v \rangle = \langle(L-a)u, v\rangle - \int_\Omega \tilde{f}(t,x, u)v \rho\textrm{d}t \textrm{d}x, \ \ \forall u, \ v\in E.
\end{eqnarray}
Then $u$ is a weak solution of problem \eqref{eqa:2-1}--\eqref{eqa:2-3} if and only if $\Phi'( u) = 0$. Therefore, solutions of problem \eqref{eqa:2-1}--\eqref{eqa:2-3} are characterized as critical points of the functional $\Phi$. Since $\tilde{f}$ is $C^1$, then $\Phi$ is a $C^2$ functional on $E$, and
\begin{equation*}
\langle \Phi''(u)w, v \rangle = \langle (L-a)w, v\rangle - \int_\Omega \frac{\partial \tilde{f}}{\partial u}(t,x, u)v w \rho\textrm{d}t \textrm{d}x, \ \ \forall u, \ v, \ w \in E.
\end{equation*}
In particular,
\begin{equation}\label{eqa:2.11}
\langle \Phi''(u)v, v \rangle = \langle (L-a)v, v\rangle - \int_\Omega \frac{\partial \tilde{f}}{\partial u}(t,x, u)v^2 \rho\textrm{d}t \textrm{d}x, \ \ \forall  u, \ v\in E.
\end{equation}

 It's not difficult to see that $\Phi$ is neither bounded from above nor from below, which shows that we can't obtain the critical points of $\Phi$ by a simple minimization or maximization. Here we shall prove our main result by virtue of the following saddle point reduction technique developed by Amann \cite{Amann.(1979)}, Castro and Lazer \cite{Castro.(1979)}, and Arcoya and Costa \cite{A.(1995)} etc.
\begin{lemma}\label{lem:3.a}
Let $H$ be a real Hilbert space with the norm $\|\cdot\|_H$, $\Phi \in C^1(H, \mathbb{R})$, and $H_1$, $H_2$ and $H_3$ be closed subset of $H$ such that $H =H_1\oplus H_2 \oplus H_3$. If there exists a constant $\gamma >0$ satisfying
$$\langle \Phi'( u+w+v_1)-\Phi'( u+w+v_2), v_1 -v_2 \rangle \leq -\gamma \|v_1 -v_2\|^2_H, \ \forall v_1, v_2 \in H_1, u\in H_2, w\in H_3,$$
and
$$\langle \Phi'( u+w_1+v)-\Phi'(  u+w_2+v), w_1 -w_2 \rangle \geq  \gamma \|w_1 -w_2\|^2_H, \ \forall v\in H_1, u\in H_2, w_1, w_2 \in H_3.$$
Then \\
(i) There exists a unique continuous mapping $h: H_2 \rightarrow H_1 \oplus H_3$, such that
$$\Phi(u + h(u)) = \max_{v\in H_1} \min_{w \in H_3} \Phi(u + v+w) = \min_{w \in H_3}\max_{v\in H_1}  \Phi(u + v+w);$$
(ii) Define $\widetilde{\Phi}(u) = \Phi(u + h(u))$ for any $u \in H_2$, then $\widetilde{\Phi} \in C^1(H_2, \mathbb{R})$, and
$$\langle \widetilde{\Phi}'( u), v \rangle = \langle \Phi'( u +  h(u)), v \rangle, \forall u, v\in H_2;$$
(iii) If $u \in H_2$ is a critical point of $\widetilde{\Phi}$, then $u + h(u)$ is a critical point of $\Phi$. On the other hand, if $u +v$ is a critical point of $\Phi$, then $v= h(u)$ and $u$ is a critical point of $\widetilde{\Phi}$, where $u \in H_2$, $v \in H_1 \oplus H_3${\upshape ;}\\
(iv) If $\Phi$ satisfies the Palais-Smale condition $(PS)_c$ at the level $c \in \mathbb{R}$, then the functional $\widetilde{\Phi}$ also satisfies the $(PS)_c$ condition.
\label{lem:3.1}
\end{lemma}

 By an observation of Lemma \ref{lem:3.a}, it needs to decompose the working space $E$ into suitable orthogonal subspaces.
 Noting $a$, $b \notin \Lambda(L)$ and $a >\rho_1$,  the working space $E$ can be decomposed into the following orthogonal subspaces
\begin{displaymath}
\begin{array}{lll}
E_1= \Big\{u \in E:  \lambda_{nm}<a,  \ m \in \mathbb{Z}, n\in \mathbb{N} \Big\},\\
E_2= \Big\{u \in E:  a<\lambda_{nm}<b, \ m \in \mathbb{Z}, n\in \mathbb{N}\Big\},\\
E_3= \Big\{u \in E:  \lambda_{nm}>b, \ m \in \mathbb{Z}, n\in \mathbb{N}\Big\}.
\end{array}
\end{displaymath}

Thus we write $E = E_1\oplus E_2 \oplus E_3$. Moreover, by $(a, b) \cap \Lambda(L) \neq \emptyset$ and Lemma \ref{lem:2.a}, we have $E_2 \neq \emptyset$ and $\dim(E_2) < \infty$.

For any $u\in E_1$, we write $u = \sum\limits_{\lambda_{nm}<a} \alpha_{nm} \varphi_n\phi_m$, then
\begin{equation}
\langle (L - a)u, u \rangle= - \sum\limits_{\lambda_{nm}<a} |\lambda_{nm} - a| |\alpha_{nm}|^2  = -\|u\|^2_E.
\label{eqa:3.1}
\end{equation}
For any $u \in E_2\oplus E_3$, by a similar calculation we have
\begin{equation}
\langle (L - a)u, u \rangle = \|u\|^2_E.
\label{eqa:3.2}
\end{equation}

With the help of \eqref{eqa:3.1} and \eqref{eqa:3.2}, we have the following lemma.
\begin{lemma}\label{lem:3.3}
Let the assumptions in Theorem {\upshape\ref{th:2.1}} hold. Then there exists a constant $\gamma > 0$ such that
$$ \langle \Phi'( u+v)-\Phi'( u+w), v -w \rangle \leq -\gamma \|v -w\|^2_E, \ \forall  u\in E_2\oplus E_3, v, w \in E_1,$$
and
$$ \langle \Phi'( u+v)-\Phi'( u+w), v -w \rangle \geq \gamma \|v -w\|^2_E, \ \forall u\in E_1\oplus E_2,  v,  w \in E_3.$$
\end{lemma}
\begin{proof}
For any  $u$, $v$, $w \in E$ and $s \in \mathbb{R}$, we have
\begin{equation}\label{eqa:3.5}
\langle \Phi'( u+v)-\Phi'( u+w), v -w \rangle = \int^1_0  \langle \Phi''(u+w+s(v-w))(v-w),  v-w\rangle \textrm{d}s.
\end{equation}
From \eqref{eqa:2.11}, we obtain
\begin{eqnarray}\label{eqa:3.6}
&&\langle \Phi''(u+w+s(v-w))(v-w),  v-w\rangle \nonumber\\
&&= \langle (L - a )(v-w), v-w \rangle - \int_\Omega (v-w)^2\frac{\partial \tilde{f}}{\partial u}(t, x, u+w+s(v-w)) \rho\textrm{d}t \textrm{d}x.\nonumber \\
\end{eqnarray}

The assumption on $\tilde{f}$ in Theorem \ref{th:2.1} shows  $\frac{\partial \tilde{f}}{\partial u} \geq 0$. For the case $v$, $w \in E_1$, $u\in E_2\oplus E_3$,
by \eqref{eqa:3.1} we have
$$\langle \Phi'( u+v)-\Phi'( u+w), v -w \rangle \leq - \|v -w\|^2_E.$$

For the case $v$, $w \in E_3$, $u\in E_1\oplus E_2$, by $\eqref{eqa:3.2}$ we have
\begin{eqnarray}\label{eqa:3.7}
\langle (L - a)( v-w), v-w \rangle = \|v-w\|^2_E.
\end{eqnarray}
Moreover, since $0 \leq \frac{\partial \tilde{f}}{\partial u}(t, x, u) \leq e - \kappa$, with the aid of \eqref{eqa:2.6}, a direct calculation yields
\begin{equation}\label{eqa:3.8}
\int_\Omega ( v-w)^2\frac{\partial \tilde{f}}{\partial u} \rho\textrm{d}t \textrm{d}x \leq (e - \kappa ) \|v-w\|^2_{L^2(\Omega)}
\leq \frac{(e-\kappa)}{e} \|v-w\|^2_E.
\end{equation}

Inserting \eqref{eqa:3.7}, \eqref{eqa:3.8} into \eqref{eqa:3.6}, from \eqref{eqa:3.5} we have
$$\langle \Phi'( u+v)-\Phi'( u+w), v -w \rangle \geq \frac{\kappa}{e} \|v-w\|^2_E.$$
By setting $\gamma = \min\{1, \ \frac{\kappa}{e}\}$, we arrive at the assertion.\qedhere
\end{proof}

Lemma \ref{lem:3.3} verifies all the conditions in Lemma \ref{lem:3.1}. Therefore there exists a unique continuous mapping $h : E_2 \rightarrow E_1\oplus E_3$ such that
\begin{equation} \label{eqa:3.9}
\widetilde{\Phi}(u) = \Phi(u + h(u)) = \max_{v\in E_1} \min_{w \in E_3} \Phi(u + v+w)= \min_{w \in E_3}\max_{v\in E_1}  \Phi(u + v+w).
\end{equation}
Moreover, Lemma \ref{lem:3.1} shows the reduction functional $\widetilde{\Phi} \in C^1(E_2, \mathbb{R})$, and
$u \in E_2$ is a critical point of $\widetilde{\Phi}$ if and only if $u+ v+w$ is a critical point of $\Phi$, where $v\in E_1$, $w\in E_3$ and $h(u) = v+w$. Thus, the critical points of $\Phi$ on the infinite dimensional space $E$ are
transformed into the ones of $\widetilde{\Phi}$ on the finite dimensional subspace $E_2$.
In what follows, we shall apply the variational methods to obtain  critical points of the functional $\widetilde{\Phi}$ on $E_2$.

\section{Verification the $(PS)_c$ condition}

\setcounter{equation}{0}
\label{sec:4}
We will acquire the critical points of $\widetilde{\Phi}$ via variational methods, thus it needs to verify  $\widetilde{\Phi}$ satisfies $(PS)_c$ condition for any $c \in \mathbb{R}$. Furthermore, Lemma \ref{lem:3.1} shows that it suffices to verify that $\Phi$ satisfies $(PS)_c$ condition, which means, any sequence $\{u_i\} \subset E$ satisfying $\Phi (u_i) \rightarrow c$ and $\Phi' (u_i) \rightarrow 0$ as $i \rightarrow \infty$ has a convergent subsequence for any $c \in \mathbb{R}$. To this goal, we need the following lemma which provides two estimates for the quadratic forms on different subspaces of $E$.
\begin{lemma}
\label{lem:3.2}
Let $a$, $b$ satisfy the assumptions in Theorem {\upshape\ref{th:2.1}}. Then there exist $\gamma_1, \gamma_2>0$ such that
\begin{eqnarray}
&&\langle (L - a - b)u, u \rangle \leq -\gamma_1\|u\|^2_E, \ \forall u\in E_1\oplus E_2, \label{eqa:3.3}\\
&&\langle (L - a - b)u, u \rangle \geq \gamma_2\|u\|^2_E, \ \forall u\in E_3.\label{eqa:3.4}
\end{eqnarray}
\end{lemma}

\begin{proof}
On the one hand, for  $u \in E_1\oplus E_2$, we write $u = \sum\limits_{\lambda_{nm}<b} \alpha_{nm} \phi_m\varphi_n$. A direct calculation yields
\begin{eqnarray*}
&&\langle (L - a - b)u, u \rangle \\
&=& \sum\limits_{\lambda_{nm}<b} (\lambda_{nm} - a ) |\alpha_{nm}|^2 - b \sum\limits_{\lambda_{nm}<b}  |\alpha_{nm}|^2\\
&\leq& \sum\limits_{\lambda_{nm}<a} (\lambda_{nm} - a) |\alpha_{nm}|^2 + \sum\limits_{a<\lambda_{nm}<b} (\lambda_{nm} - a ) |\alpha_{nm}|^2- b \sum\limits_{a<\lambda_{nm}<b}  |\alpha_{nm}|^2.
\end{eqnarray*}
In virtue of $a<\lambda_{nm}<b$, we have $0<\lambda_{nm} - a<b^-$. Thus, we have
$$\frac{1}{b^-} \sum\limits_{a<\lambda_{nm}<b}  |\lambda_{nm} - a | |\alpha_{nm}|^2 \leq \sum\limits_{a<\lambda_{nm}<b}  |\alpha_{nm}|^2.$$
Therefore,
\begin{eqnarray*}
&&\langle (L - a - b)u, u \rangle\\
 &\leq& -\sum\limits_{\lambda_{nm}<a} |\lambda_{nm} - a | |\alpha_{nm}|^2-\left(\frac{b}{b^-}-1\right)\sum\limits_{a<\lambda_{nm}<b}  |\lambda_{nm} - a | |\alpha_{nm}|^2\\
 &\leq&-\gamma_1\|u\|^2_E,
\end{eqnarray*}
where $\gamma_1 = \min \{1,  \frac{b}{b^-} -1\}$. Noting $0<b^- < b$, it follows $\gamma_1>0$.

On the other hand, for $u\in E_3$,  we write $u = \sum\limits_{\lambda_{nm}>b} \alpha_{nm} \phi_m\varphi_n$. Observing $\lambda_{nm} > b$, from \eqref{eqa:2.6}, we obtain
$$\frac{1}{e} \sum\limits_{\lambda_{nm}>b}  |\lambda_{nm} - a | |\alpha_{nm}|^2 \geq \sum\limits_{\lambda_{nm}>b}  |\alpha_{nm}|^2.$$
Therefore,
\begin{eqnarray*}
\langle (L - a - b)u, u \rangle &=& \sum\limits_{\lambda_{nm}>b} |\lambda_{nm} - a | |\alpha_{nm}|^2 - b \sum\limits_{\lambda_{nm}>b} |\alpha_{nm}|^2\\
&\geq& \left(1-\frac{b}{e} \right) \sum\limits_{\lambda_{nm}>b} |\lambda_{nm} - a | |\alpha_{nm}|^2.
\end{eqnarray*}
Let $\gamma_2 =  1-\frac{b}{e}$, by \eqref{eqa:2.5}, it follows $\gamma_2>0$. Therefore
$$\langle (L - a - b)u, u \rangle \geq \gamma_2\|u\|^2_E.$$
The proof is completed.\qedhere
\end{proof}
\begin{lemma}\label{lem:4.1} Let the assumptions in Theorem {\upshape\ref{th:2.1}} hold. If  $\{u_i\} \subset E$ satisfies $\Phi (u_i) \rightarrow c$ and $\Phi' (u_i) \rightarrow 0$ as $i \rightarrow \infty$, then there exists a constant $\widetilde{C} > 0$ independent of $i$  such that $\|u_i\|_E \leq \widetilde{C}$.
\end{lemma}
\begin{proof}
Split $u_i = u^+_i +u^-_i$, where $u^+_i \in E_3$ and $u^-_i \in E_1\oplus E_2$, $i= 1,2, \cdots$.

(i) For $u^+_i \in E_3$, since $\Phi' (u_i) \rightarrow 0$ as $i \rightarrow \infty$, from \eqref{eqa:2.10}, we have
\begin{eqnarray}\label{eqa:4.1}
o(1)\|u^+_i\|_E &\geq & \langle \Phi'( u_i), u^+_i \rangle = \langle (L - a)u^+_i, u^+_i \rangle - \int_\Omega \tilde{f}(t,x, u_i)u^+_i \rho\textrm{d}t \textrm{d}x \nonumber\\
&=& \langle (L - a - b)u^+_i, u^+_i \rangle - \int_\Omega (\tilde{f}(t,x, u_i)- b u_i)u^+_i \rho\textrm{d}t \textrm{d}x.\ \ \ \ \ \ \
\end{eqnarray}
From \eqref{eqa:3.4}, we obtain
\begin{equation}\label{eqa:4.2}
\langle (L - a - b)u^+_i, u^+_i \rangle \geq \gamma_2\|u^+_i\|^2_E.
\end{equation}
On the other side, the condition \eqref{eqa:1.3} shows that for any $\varepsilon > 0$, there exists a constant $C =C(\varepsilon) > 0$ such that
\begin{eqnarray}\label{eqa:4.3}
|\tilde{f}(t, x, u_i) - b u_i| < \varepsilon |u_i| + C.
\end{eqnarray}
Therefore, by H\"{o}lder inequality and \eqref{eqa:2.7}, \eqref{eqa:2.8}, a direct calculation yields
\begin{eqnarray}\label{eqa:4.4}
&& \Big|\int_\Omega (\tilde{f}(t,x, u_i)- b u_i)u^+_i \rho\textrm{d}t \textrm{d}x \Big| \nonumber\\
&\leq & \varepsilon \|u^+_i\|_{L^2(\Omega)}\|u_i\|_{L^2(\Omega)} + C \|u_i^+\|_{L^1(\Omega)} \nonumber\\
&\leq &\frac{\varepsilon}{2\delta} \|u^+_i\|^2_E + \frac{\varepsilon}{2\delta} \|u_i\|^2_E + C \|u_i^+\|_E,
\end{eqnarray}
for some constant $C$ independent of $i$.

Inserting \eqref{eqa:4.2}, \eqref{eqa:4.4} into \eqref{eqa:4.1},  we obtain
\begin{equation}\label{eqa:4.5}
(\gamma_2 - \frac{\varepsilon}{2\delta})\|u^+_i\|^2_E - \frac{\varepsilon}{2\delta}\|u_i\|^2_E - C\|u_i^+\|_E \leq 0.
\end{equation}

(ii) For $u^-_i \in E_1\oplus E_2$, by \eqref{eqa:3.3}, we have $\langle (L - a - b)u^-_i, u^-_i \rangle \leq -\gamma_1\|u^-_i\|^2_E$.
Moreover, noting that
\begin{eqnarray*}
o(1)\|u^-_i\|_E &\geq & \langle -\Phi'( u_i), u^-_i \rangle  \\
&=& -\langle (L - a - b)u^-_i, u^-_i \rangle + \int_\Omega (\tilde{f}(t,x, u_i)- b u_i)u^-_i \rho\textrm{d}t \textrm{d}x,
\end{eqnarray*}
a similar calculation as in \eqref{eqa:4.4} yields
\begin{eqnarray}\label{eqa:4.6}
\Big|\int_\Omega (\tilde{f}(t,x, u_i)- b u_i)u^-_i \rho\textrm{d}t \textrm{d}x \Big|
\leq \frac{\varepsilon}{2\delta} \|u^-_i\|^2_E + \frac{\varepsilon}{2\delta} \|u_i\|^2_E + C \|u_i^-\|_E.
\end{eqnarray}
Consequently,
\begin{eqnarray}\label{eqa:4.7}
(\gamma_1 - \frac{\varepsilon}{2\delta})\|u^-_i\|^2_E - \frac{\varepsilon}{2\delta}\|u_i\|^2_E - C\|u_i^-\|_E \leq 0.
\end{eqnarray}

Let $\gamma_0 = \min\{\gamma_1,\gamma_2\}$. Since $E_1$, $ E_2$ and $E_3$ are orthogonal subspaces of $E$, we have $\|u_i\|^2_E = \|u^+_i\|^2_E + \|u^-_i\|^2_E$. Therefore, by the fact $\|u^+_i\|_E + \|u^-_i\|_E \leq\sqrt{2} \|u_i\|_E$, the sum of \eqref{eqa:4.5} and \eqref{eqa:4.7} yields
\begin{eqnarray}\label{eqa:4.8}
(\gamma_0 - \frac{3\varepsilon}{2\delta})\|u_i\|^2_E - C\|u_i\|_E \leq 0.
\end{eqnarray}
Taking $\varepsilon \in (0, \frac{2\delta \gamma_0}{3})$ in \eqref{eqa:4.8}, we obtain there exists a constant $\widetilde{C} > 0$ independent of $i$ such that $\|u_i\|_E \leq \widetilde{C}$. We arrive at the result.
\end{proof}

Since $E = E_1\oplus E_2 \oplus E_3$, we can rewrite $E = E_1 \oplus E_1^\bot$ for simplicity, where
$$E_1^\bot = E_2\oplus E_3 = \Big\{u \in E: \lambda_{nm}>a,  \ m \in \mathbb{Z}, n\in \mathbb{N} \Big\}.$$
Denote
$$E_0 = \Big\{u \in L^2(\Omega) : \ \  p n = q|m|,  \ m \in \mathbb{Z}, n\in \mathbb{N}\Big\}.$$
\begin{remark}\label{rem:4.2}
Under the assumptions  of Theorem {\upshape\ref{th:2.1}}, the Lemma {\upshape\ref{lem:2.a}} shows $E_0$ is an infinite dimensional subspace spanned by the eigenfunctions $\phi_m\varphi_n$ for $p n = q|m|$. Moreover, it is easy to see that $\dim (E_1^\bot \cap E_0) <\infty$ and $\dim (E_1 \cap E_0) =\infty$.
\end{remark}
\begin{proposition}\label{pp:4.1}
For $1\leq r\leq 2$, the embedding
\begin{equation}\label{eqa:4.9}
E\ominus E_0 \hookrightarrow L^r(\Omega),
\end{equation}
is compact.
\end{proposition}
\begin{proof}
For $u \in E\ominus E_0$, it can be expanded as $u = \sum\limits_{n,m} \alpha_{nm} \phi_m\varphi_n$ for $pn \neq q|m|$.

The fact $\|u\|_E = \big(\sum\limits_{n,m}|\lambda_{nm} - a| |\alpha_{nm}|^2\big)^\frac{1}{2}$ shows that the mapping
$$I_0: u = \sum\limits_{n,m} \alpha_{nm} \phi_m\varphi_n \mapsto \{|\lambda_{nm} - a|^\frac{1}{2} \alpha_{nm}\}$$
is continuous from $E\ominus E_0$ to $l^2$.

Since $|\lambda_{nm} - a|\rightarrow \infty$ as $m, n \rightarrow \infty$, it follows that the mapping
$$I_1:\{|\lambda_{nm} - a|^\frac{1}{2}\alpha_{nm}\} \mapsto \{\alpha_{nm}\}$$
is compact from $l^2$ to $l^2$.

Since $\phi_m\varphi_n$ is a complete orthonormal sequence of $L^2(\Omega)$, then the mapping
$$I_2:\{\alpha_{nm}\} \mapsto u= \sum\limits_{n,m} \alpha_{nm} \phi_m\varphi_n$$
is continuous from $l^2$ to $L^2(\Omega)$.

Consequently, the mapping
$$ I_2I_1I_0:E\ominus E_0 \rightarrow L^2(\Omega)$$
is compact. Furthermore, for $1\leq r \leq 2$, the continuous embedding $L^2(\Omega) \hookrightarrow L^r(\Omega)$ implies that the embedding $E\ominus E_0\hookrightarrow L^r(\Omega)$ is compact.
\end{proof}

\begin{lemma}\label{lem:4.2}
Let the assumptions of Theorem {\upshape\ref{th:2.1}} hold. Then $\Phi$ satisfies the $(PS)_c$ condition for any $c \in \mathbb{R}$.
\end{lemma}
\begin{proof}
For any $c \in \mathbb{R}$, assume $\{u_i\} \subset E$ satisfies $\Phi (u_i) \rightarrow c$ and $\Phi' (u_i) \rightarrow 0$ as $i \rightarrow \infty$.
Since $E$ is a Hilbert space, by Lemma \ref{lem:4.1}, we have $u_i \rightharpoonup u$ as $i \rightarrow \infty$ for some $u \in E$.
Decompose $u_i = v_i + y_i + w_i + z_i$ and $u = v + y + w + z$,  where $v$, $y$, $w$, $z$ are the weak limits of $\{v_i\}$,  $\{y_i\}$, $\{w_i\}$, $\{z_i\}$ respectively, and $v_i, v \in E_1^\bot \ominus E_0$, $y_i, y \in E_1^\bot \cap E_0$, $w_i, w\in E_1 \ominus E_0$, $z_i, z \in E_1 \cap E_0$.

(i) For $v_i, v \in E_1^\bot \ominus E_0$, from \eqref{eqa:3.2} we have
\begin{eqnarray}\label{eqa:4.10}
\|v_i-v\|^2_E &=& \langle (L - a)( v_i-v), v_i-v \rangle \nonumber\\
&=& \langle (L - a) v_i, v_i-v \rangle -\langle (L - a) v, v_i-v \rangle.
\end{eqnarray}
In virtue of $v_i \rightharpoonup v$ in $E$ and $E \hookrightarrow L^2(\Omega)$, we have $v_i \rightharpoonup v$ in $L^2(\Omega)$ along with a  subsequence of $\{v_i\}$. In fact, by Proposition \ref{pp:4.1}, we also have $v_i \rightarrow v$ in ${L^2(\Omega)}$. We still use $\{v_i\}$ to denote the  subsequence for convenience. Thus, it follows
\begin{eqnarray}\label{eqa:4.11}
\langle (L - a) v, v_i-v \rangle \rightarrow 0, \ \ {\rm as} \ \ i \rightarrow \infty.
\end{eqnarray}

On the other hand, noting $v_i, v \in E_1^\bot \ominus E_0$ and $u_i = v_i + y_i + w_i + z_i$, we have $u_i - v_i\in (E_1^\bot \ominus E_0)^\perp$. Thus
$$\langle (L - a)( u_i-v_i), v_i-v \rangle = 0.$$
By \eqref{eqa:2.10}, we have
\begin{eqnarray}\label{eqa:4.12}
&&\langle (L - a) v_i, v_i-v \rangle = \langle (L - a) u_i, v_i-v \rangle \nonumber\\
&=& \langle \Phi'(u_i), v_i-v \rangle + \int_\Omega \tilde{f}(t,x, u_i)(v_i-v) \rho\textrm{d}t \textrm{d}x.
\end{eqnarray}
Since $\Phi' (u_i) \rightarrow 0$ as $i \rightarrow \infty$, we have
\begin{eqnarray}\label{eqa:4.13}
\langle \Phi'(u_i), v_i-v \rangle \rightarrow 0, \ \textrm{as} \ \ i \rightarrow \infty.
\end{eqnarray}
By \eqref{eqa:4.3} and H\"{o}lder inequality, a direct calculation yields
\begin{eqnarray*}
\Big|\int_\Omega \tilde{f}(t,x, u_i)(v_i-v) \rho\textrm{d}t \textrm{d}x \Big|
\leq  (b + \varepsilon ) \|u_i\|_{L^2(\Omega)}\|v_i-v\|_{L^2(\Omega)}+ C\|v_i-v\|_{L^1(\Omega)}.
\end{eqnarray*}

In virtue of $v_i \rightarrow v$ in ${L^2(\Omega)}$ and the continuous embedding $L^2(\Omega, \rho)  \hookrightarrow L^1(\Omega)$, we obtain  $v_i \rightarrow v$ in ${L^1(\Omega)}$. Therefore,
\begin{equation}\label{eqa:4.14}
\Big|\int_\Omega \tilde{f}(t,x, u_i)(v_i-v) \rho\textrm{d}t \textrm{d}x \Big| \rightarrow 0, \ \textrm{as} \ \ i \rightarrow \infty.
\end{equation}
Inserting \eqref{eqa:4.13}, \eqref{eqa:4.14} into \eqref{eqa:4.12}, we have
\begin{eqnarray}\label{eqa:4.15}
\langle (L - a) v_i, v_i-v \rangle \rightarrow 0, \ \textrm{as} \ \ i \rightarrow \infty.
\end{eqnarray}
Finally, substituting \eqref{eqa:4.15}, \eqref{eqa:4.11} into \eqref{eqa:4.10}, we have
\begin{eqnarray}\label{eqa:4.16}
\|v_i-v\|_E \rightarrow 0,  \ \textrm{as} \ \ i \rightarrow \infty.
\end{eqnarray}

(ii) For $y_i, y \in E_1^\bot \cap E_0$, since $\dim (E_1^\bot \cap E_0) <\infty$ and $y_i \rightharpoonup y$ in $E$, then there exists a subsequence of $\{y_i\}$ which strongly converges to $y$ in $E$. The subsequence is still denoted by $\{y_i\}$.

(iii) For $w_i, w\in E_1 \ominus E_0$, by $\eqref{eqa:2.10}$ and $\eqref{eqa:3.1}$ we obtain
\begin{eqnarray*}
\|w_i-w\|^2_E &=& - \langle (L - a)(w_i-w), w_i-w \rangle \\
&=& - \langle \Phi'(u_i), w_i-w \rangle - \int_\Omega \tilde{f}(t,x, u_i)(w_i-w) \rho\textrm{d}t \textrm{d}x + \langle (L - a) w, w_i-w \rangle.
\end{eqnarray*}
A similar calculation as in (i) yields
\begin{equation}\label{eqa:4.17}
\|w_i-w\|_E \rightarrow 0,  \ \textrm{as} \ \ i \rightarrow \infty.
\end{equation}

(iv) Since $z_i\in E_1 \cap E_0$, then the compact embedding \eqref{eqa:4.9} is invalid. In addition, since $\dim (E_1 \cap E_0) = \infty$,  we can not extract a strong convergence subsequence of $\{z_i\}$ similar to $\{y_i\}$. In what follows, with the aid of the monotone method, we prove $z_i \rightarrow z$ as $i \rightarrow \infty$ in $E$.

Since $\Phi' (u_i) \rightarrow 0$ and $z_i \rightharpoonup z$ in $E$, we have
\begin{eqnarray}\label{eqa:4.18}
&&\|z_i-z\|^2_E = - \langle (L - a)( z_i-z), z_i-z \rangle \nonumber\\
&=& - \langle \Phi'(u_i), z_i-z \rangle - \int_\Omega \tilde{f}(t,x, u_i)(z_i-z) \rho\textrm{d}t \textrm{d}x + \langle (L - a) z, z_i-z \rangle\nonumber\\
&\leq &  - \int_\Omega \tilde{f}(t,x, u_i)(z_i-z) \rho\textrm{d}t \textrm{d}x + o(1).
\end{eqnarray}

Denote $\tilde{f}(u_i)=\tilde{f}(t,x, u_i)$ for simplicity.  We rewrite $\int_\Omega \tilde{f}(t,x, u_i)(z_i-z) \rho\textrm{d}t \textrm{d}x$ in the form of inner product and decompose it as follows
\begin{eqnarray}\label{eqa:4.19}
&& \int_\Omega \tilde{f}(t,x, u_i)(z_i-z) \rho\textrm{d}t \textrm{d}x\\
& =& \langle \tilde{f}(u_i), z_i-z \rangle  \nonumber\\
&=&\langle \tilde{f}(u_i)- \tilde{f}(\tilde{u}_i + z), z_i-z \rangle + \langle \tilde{f}(\tilde{u}_i + z) - \tilde{f}(u), z_i-z \rangle\nonumber\\
 &&+ \langle \tilde{f}(u), z_i-z \rangle,
\end{eqnarray}
where $\tilde{u}_i = v_i + y_i + w_i$.

Since $\tilde{f}$ is increasing in $u$, then
\begin{equation}\label{eqa:4.20}
\langle \tilde{f}(u_i)- \tilde{f}(\tilde{u}_i + z), z_i-z \rangle \geq 0.
\end{equation}

To continue the discussion, by \eqref{eqa:4.3}, we have $\tilde{f} : u \mapsto \tilde{f}(t,x, u)$ is continuous from ${L^2(\Omega)}$ to ${L^2(\Omega)}$. Moreover, from the proof of (i)--(iii), we have $\tilde{u}_i \rightarrow \tilde{u}$ in $E$, where $\tilde{u} = v + y + w$. The inequality \eqref{eqa:2.7} shows $\tilde{u}_i \rightarrow \tilde{u}$ in $L^2(\Omega)$. Thus
\begin{equation}\label{eqa:4.21}
\langle \tilde{f}(\tilde{u}_i + z) - \tilde{f}(u), z_i-z \rangle \rightarrow 0, \ \ {\rm as} \ \ i \rightarrow \infty.
\end{equation}

Consequently, since $z_i \rightharpoonup z$ in ${L^2(\Omega)}$, by \eqref{eqa:4.18}, \eqref{eqa:4.20} and \eqref{eqa:4.21}, we have
\begin{eqnarray*}
\|z_i-z\|_E \rightarrow 0,  \ \textrm{as} \ \ i \rightarrow \infty.
\end{eqnarray*}
We arrive at result.
\end{proof}

\section{Bounds of the reduction functional}
\setcounter{equation}{0}
\label{sec:5}

The following three lemmas are concerned with the bounds of the reduction functional $\widetilde{\Phi}$ and play an important role in the proof of Theorem \ref{th:2.1}.
\begin{lemma}\label{lem:5.1}
Let the assumptions of Theorem {\upshape\ref{th:2.1}} hold. Then \\
 (i) there exists a constant $M > 0$, such that $\widetilde{\Phi}(u) < M$, $\forall u \in E_2${\upshape ;}\\
 (ii) there exists a constant $\widetilde{R} > 0$ such that $\widetilde{\Phi}(u) \leq 0$ for $u \in E_2$ with $\|u\|_E \geq \widetilde{R}$.
\end{lemma}
\begin{proof}
For $u \in E_2$, from \eqref{eqa:3.9}, we have
$$\widetilde{\Phi}(u) = \min_{w \in E_3} \max_{v\in E_1} \Phi(u + v+w) \leq \max_{v\in E_1}  \Phi(u + v),$$
where
\begin{equation}\label{eqa:5.1}
\Phi(u +v)  = \frac{1}{2}\langle(L - a - b)(u +v), u +v\rangle - \int_\Omega \left(\widetilde{F}(t,x, u +v) - \frac{b}{2}(u +v)^2\right)\rho\textrm{d}t \textrm{d}x.
\end{equation}
By \eqref{eqa:4.3}, it follows
\begin{eqnarray}\label{eqa:5.2}
|\widetilde{F}(t,x, u +v) - \frac{b}{2}(u +v)^2|
\leq   \varepsilon |u +v|^2 + C|u +v|.
\end{eqnarray}
Inserting \eqref{eqa:5.2} into \eqref{eqa:5.1}, by \eqref{eqa:3.3}, it follows
\begin{eqnarray*}
\Phi(u +v) &\leq & -\frac{\gamma_1}{2} \|u +v\|^2_E + \int_\Omega \Big(\varepsilon |u +v|^2 + C|u +v|\Big) \rho\textrm{d}t \textrm{d}x \nonumber\\
&\leq & -\frac{\gamma_1}{2} \|u +v\|^2_E  + \varepsilon \|u +v\|^2_{L^2(\Omega)} + C\|u +v\|_{L^1(\Omega)},
\end{eqnarray*}
for some constant $C$ depending on $\varepsilon$.

Taking $\varepsilon = \frac{\delta \gamma_1}{4}$ in above inequality, by \eqref{eqa:2.7} and \eqref{eqa:2.8},  we obtain
\begin{eqnarray}\label{eqa:5.3}
\Phi(u +v) &\leq &  -\frac{\gamma_1}{4} \|u +v\|^2_E + C\|u +v\|_E.
\end{eqnarray}
Therefore, the estimate \eqref{eqa:5.3} shows there exists $M > 0$ such that $\Phi(u +v) \leq M$ and the assertion (i) is proved.

 By the fact $\|u +v\|^2_E = \|u \|^2_E + \|v\|^2_E$ and $\|u +v\|_E \leq \|u \|_E + \|v\|_E$ , from  \eqref{eqa:5.3}, we have
\begin{eqnarray*}
\Phi(u +v) &\leq & -\frac{\gamma_1}{4} \|u \|^2_E + C \|u \|_E + (-\frac{\gamma_1}{4} \|v \|^2_E + C \|v \|_E)\\
&\leq &  -\frac{\gamma_1}{4} \|u \|^2_E + C \|u \|_E  + C_0,
\end{eqnarray*}
where $C_0 = \max\limits_{s \geq 0} \{-\frac{\gamma_1}{4} s^2 + C s \}$.
Thus, there exists a constant $\widetilde{R} > 0$ such that $\Phi(u +v) \leq 0$ for $\|u\|_E \geq \widetilde{R}$. We prove the assertion (ii).
\end{proof}
\begin{lemma}\label{lem:5.2}
 Let the assumptions of Theorem {\upshape\ref{th:2.1}} hold. Then for any $R_1 > 0$, there exists a constant $\tau_1$ depending on $R_1$ such that $$\widetilde{\Phi}(u) \geq \tau_1, \ \ \forall  u \in E_2 \cap B_{R_1},$$ where $B_{R_1} = \{u \in E : \|u \|_E < R_1\}$.
\end{lemma}
\begin{proof}
For $u \in E_2$, by \eqref{eqa:3.9}, we have
$$\widetilde{\Phi}(u) = \max_{v\in E_1} \min_{w \in E_3} \Phi(u + v+w) \geq \min_{w \in E_3}  \Phi(u + w),$$
where
\begin{equation}\label{eqa:5.4}
\Phi(u +w)  = \frac{1}{2}\langle(L- a - b)(u +w), u +w\rangle
 - \int_\Omega \left(\widetilde{F}(t,x, u +w) - \frac{b}{2}(u +w)^2\right)\rho\textrm{d}t \textrm{d}x.
\end{equation}

Since $E_2$ and $E_3$ are orthogonal subspaces of $E$, we have
\begin{equation*}
\langle(L-a - b)(u +w), u +w\rangle = \langle(L-a - b)u, u \rangle + \langle(L-a-b)w, w\rangle.
\end{equation*}

By \eqref{eqa:3.4}, we have $\langle(L-a - b)w, w\rangle \geq \gamma_2 \|w \|^2_E$.
Moreover, noting $\langle (L - a)u, u \rangle = \|u\|^2_E$,
by \eqref{eqa:2.7}, we obtain
\begin{equation*}
|\langle(L-a - b)u, u \rangle|
\leq  \|u \|^2_E + b  \|u \|^2_{L^2(\Omega)}
  \leq  (1 + \frac{b}{\delta}) \|u \|^2_E = C_1 \|u \|^2_E,
\end{equation*}
where $C_1 = 1 + \frac{b}{\delta}$. Thus, we have
\begin{eqnarray}\label{eqa:5.5}
\langle(L-a - b)(u +w), u +w\rangle \geq \gamma_2 \|w \|^2_E - C_1 \|u \|^2_E.
\end{eqnarray}

On the other hand, a similar calculations as in \eqref{eqa:5.2} yields
\begin{eqnarray}\label{eqa:5.6}
|\widetilde{F}(t,x, u +w) - \frac{b}{2}(u +w)^2|\leq \varepsilon |u +w|^2 + C|u +w|.
\end{eqnarray}

Inserting \eqref{eqa:5.5}, \eqref{eqa:5.6} into \eqref{eqa:5.4} and taking $\varepsilon = \frac{\delta \gamma_2}{4}$, from \eqref{eqa:2.7} and \eqref{eqa:2.8}, we have
\begin{eqnarray*}
\Phi(u +w) &\geq & \frac{\gamma_2}{2} \|w \|^2_E - \frac{C_1}{2} \|u \|^2_E -\frac{\delta \gamma_2}{4} \|u +w\|^2_{L^2(\Omega)} - C\|u +w\|_{L^1(\Omega)} \nonumber\\
&\geq & \frac{\gamma_2}{2} \|w \|^2_E - \frac{C_1}{2} \|u \|^2_E - \frac{\gamma_2}{4}(\|u \|^2_E + \|w \|^2_E)- C(\|u \|_E + \|w \|_E)\nonumber\\
&\geq & -(\frac{C_1}{2} + \frac{\gamma_2}{4}) \|u \|^2_E - C \|u \|_E + C_2,
\end{eqnarray*}
where $C_2 = \min\limits_{s \geq 0} \{\frac{\gamma_2}{4} s^2 - C s\}$.

For any $R_1>0$,  the above estimate shows  $\widetilde{\Phi}(u) \geq \tau_1= -(\frac{C_1}{2} + \frac{\gamma_2}{4}) R_1^2 - C R_1 + C_2$ for any $\|u \|_E < R_1$. The proof is completed.
\end{proof}
\begin{lemma}\label{lem:5.3}
Let the assumptions of Theorem {\upshape\ref{th:2.1}} hold. Then there exist two constants $\tau_2 > 0$ and $R_2 > 0$ such that $\widetilde{\Phi}(u) \geq \tau_2$, for any $u \in E_2$ with $\|u\|_E = R_2$.
\end{lemma}
\begin{proof}
Firstly, for any $u \in E_2$, $w \in E_3$, by \eqref{eqa:2.9}, we have
\begin{eqnarray}\label{eqa:5.7}
\Phi(u +w)  = \frac{1}{2}\langle(L-a )(u +w), u +w\rangle - \int_\Omega \widetilde{F}(t,x, u +w) \rho\textrm{d}t \textrm{d}x.
\end{eqnarray}

Since $E_2$ and $E_3$ are orthogonal subspaces of $E$, from \eqref{eqa:3.2}, we obtain
\begin{eqnarray}\label{eqa:5.8}
\langle(L-a )(u +w), u +w\rangle = \|u \|^2_E + \|w \|^2_E.
\end{eqnarray}

In addition, it is easy to see
\begin{eqnarray*}
\int^1_0 \int^1_0 s \frac{\partial \tilde{f}}{\partial \xi} (u + s\theta w) w^2 \textrm{d} \theta \textrm{d} s &=& \int^1_0 w \tilde{f}(u + s w)\textrm{d}s  - \tilde{f}(u)w \\
&=&  \widetilde{F}(u + w) - \widetilde{F}(u ) -\tilde{f}(u)w,
\end{eqnarray*}
where $\tilde{f}(\xi) = \tilde{f}(t,x, \xi)$ and $\widetilde{F}(\xi) = \widetilde{F}(t,x, \xi)$.
Thus,
\begin{eqnarray}\label{eqa:5.9}
\widetilde{F}(u + w) = \int^1_0 \int^1_0 s \frac{\partial \tilde{f}}{\partial \xi} (u + s\theta w) w^2 \textrm{d} \theta \textrm{d}s + \tilde{f}(u)w + \widetilde{F}(u ).
\end{eqnarray}
In what follows, we estimate the upper bound of $\int_\Omega \widetilde{F}(u + w) \rho\textrm{d}t \textrm{d}x$. The equation \eqref{eqa:5.9} shows that it needs to estimate the upper bounds of the following three terms:
$$\int_\Omega \left(\int^1_0 \int^1_0 s \frac{\partial \tilde{f}}{\partial \xi} (u + s\theta w) w^2 \textrm{d} \theta \textrm{d}s \right) \rho\textrm{d}t \textrm{d}x, \ \int_\Omega \tilde{f}(u)w \rho\textrm{d}t \textrm{d}x, \ \int_\Omega \widetilde{F}(u) \rho\textrm{d}t \textrm{d}x.$$

(i)  For $w \in E_3$, the fact $|\lambda_{nm} - a|\geq e$ shows $\|w \|^2_{L^2(\Omega)} \leq \frac{1}{e} \|w \|^2_E$. Observing $0\leq \frac{\partial \tilde{f}}{\partial u}(t, x, u) \leq e - \kappa$, we obtain
\begin{eqnarray}\label{eqa:5.11}
0\leq \int_\Omega  \left(\int^1_0 \int^1_0 s \frac{\partial \tilde{f}}{\partial \xi} (u + s\theta w) w^2 \textrm{d} \theta \textrm{d}s \right) \rho\textrm{d}t \textrm{d}x  \leq \frac{e - \kappa}{2 e} \|w \|^2_E.
\end{eqnarray}

(ii) Fix $r > 1$, by the assumptions \eqref{eqa:1.2} and \eqref{eqa:1.3}, we obtain that for any $\varepsilon >0$, there exists a constant $C=C(\varepsilon, r) >0$ such that
\begin{eqnarray}\label{eqa:5.12}
|\tilde{f}(u)| \leq \varepsilon |u| + C |u|^r, \quad \forall \ (t, x, u) \in \Omega \times \mathbb{R}.
\end{eqnarray}
Therefore, a direct calculation yields
\begin{eqnarray*}
\Big| \int_\Omega \tilde{f}(u)w \rho\textrm{d}t \textrm{d}x \Big|
&\leq& \varepsilon \|u \|_{L^2(\Omega)} \|w \|_{L^2(\Omega)} + C \|u \|^r_{L^{2r}(\Omega)}\|w \|_{L^2(\Omega)} \nonumber\\
&\leq& \frac{\varepsilon}{2} \|u \|^2_{L^2(\Omega)}  + C \|u \|^{2r}_{L^{2r}(\Omega)} + \varepsilon\|w \|^2_{L^2(\Omega)}.
\end{eqnarray*}

Since $\dim(E_2) < \infty$, then all norms of $E_2$ are equivalent. Thus there exists a constant $C>0$ such that $\|u \|^{2r}_{L^{2r}(\Omega)} \leq C \|u \|^{2r}_E$. Since $\|w \|^2_{L^2(\Omega)} \leq \frac{1}{e} \|w \|^2_E$, by \eqref{eqa:2.7}, we have
\begin{eqnarray}\label{eqa:5.13}
\Big| \int_\Omega \tilde{f}(u)w \rho\textrm{d}t \textrm{d}x \Big| \leq \frac{\varepsilon}{2\delta} \|u \|^2_E  + C \|u \|^{2r}_E + \frac{\varepsilon}{e}\|w \|^2_E.
\end{eqnarray}

(iii) Since $\dim(E_2) < \infty$, there exists a constant $C>0$ such that $\|u \|^{r+1}_{L^{r+1}(\Omega)} \leq C \|u \|^{r+1}_E$. Thus, by \eqref{eqa:2.7} and \eqref{eqa:5.12}, a direct calculation yields
\begin{equation}\label{eqa:5.14}
\Big|\int_\Omega \widetilde{F}(u) \rho\textrm{d}t \textrm{d}x \Big|
\leq  \frac{\varepsilon}{2\delta} \|u \|^2_E + C \|u \|^{r+1}_E,
\end{equation}
for some constant $C$ depending on $\varepsilon$ and $r$.

Therefore, the sum of \eqref{eqa:5.11}, \eqref{eqa:5.13} and \eqref{eqa:5.14} yields
\begin{eqnarray}\label{eqa:5.15}
\Big|\int_\Omega \widetilde{F}(u + w)\rho\textrm{d}t \textrm{d}x \Big| &\leq& \frac{\varepsilon}{\delta} \|u \|^2_E + C \|u \|^{r+1}_E + C \|u \|^{2r}_E \nonumber \\
&&+ \frac{e - \kappa + 2\varepsilon}{2 e} \|w \|^2_E.
\end{eqnarray}

Consequently, inserting \eqref{eqa:5.8} and \eqref{eqa:5.15} into \eqref{eqa:5.7} and taking $\varepsilon = \min\{\frac{\delta}{4}, \frac{\kappa}{4}\}$, we have
\begin{eqnarray}\label{eqa:5.16}
\Phi(u +w) &\geq& (\frac{1}{2} - \frac{\varepsilon}{\delta}) \|u \|^2_E - C \|u \|^{r+1}_E - C \|u \|^{2r}_E + (\frac{1}{2} - \frac{e - \kappa + 2\varepsilon}{2 e}) \|w \|^2_E \nonumber\\
&=& (\frac{1}{2} - \frac{\varepsilon}{\delta}) \|u \|^2_E  - C \|u \|^{r+1}_E - C \|u \|^{2r}_E + ( \frac{ \kappa - 2\varepsilon}{2 e}) \|w \|^2_E.\nonumber\\
&\geq &  \frac{1}{4}\|u \|^2_E - C \|u \|^{r+1}_E - C \|u \|^{2r}_E.
\end{eqnarray}

Finally, since $r > 1$, then  $\phi(s)= \frac{1}{4}s^2 - C s^{r+1} - C s^{2r}$  attains the local minimum at $s = 0$. Therefore there exist  two constants $R_2>0$ and $\tau_2 > 0$ such that $\phi (R_2) \geq \tau_2$.
Since  $\widetilde{\Phi}(u)  \geq \min\limits_{w \in E_3}  \Phi(u + w)$, from \eqref{eqa:5.16}, we obtain
$\widetilde{\Phi}(u)  \geq \tau_2 $ with $\|u \|_E = R_2$.
\end{proof}

\section{Proof of Theorem \ref{th:2.1}}

\setcounter{equation}{0}
\label{sec:6}

With the above lemmas in hand, we give the proof of Theorem {\upshape\ref{th:2.1}}.

\begin{proof}
Let $R_2$ and $\tau_2$ be the constants in Lemma \ref{lem:5.3} and $B_{R_2} = \{u \in E_2 : \|u \|_E < R_2\}$.

Firstly, by Lemma \ref{lem:5.1}, we know $\widetilde{\Phi} $  is bounded from above. Let $\beta_1 = \sup\limits_{u \in E_2}\widetilde{\Phi}(u)$.  By Lemma \ref{lem:3.1} and Lemma \ref{lem:4.2}, $\widetilde{\Phi} $ satisfies $(PS)_{\beta_1}$ condition. Thus, there exists a critical point $u_1 \in  E_2$ such that $\widetilde{\Phi} (u_1)= \beta_1$ and $\widetilde{\Phi}' (u_1)= 0$.

Secondly, recall $\widetilde{F}(t, r, u)\geq 0$ for any $u \in E$, then for any $v \in E_1$, from \eqref{eqa:3.1}, we obtain
$$\Phi(v)  = \frac{1}{2}\langle(L-a)v, v\rangle - \int_\Omega \widetilde{F}(t,x, v) \rho\textrm{d}t \textrm{d}x\leq 0.$$
Therefore, we have
\begin{equation}\label{eqa:6.1}
\widetilde{\Phi}(0) = \min_{w \in E_3} \max_{v\in E_1} \Phi(v+w) \leq \max_{v\in E_1}  \Phi(v) \leq 0.
\end{equation}
Noting $0 \in B_{R_2}$ and taking $R_1 = R_2$ in Lemma \ref{lem:5.2}, in virtue of Lemma \ref{lem:5.2}, Lemma \ref{lem:5.3} and $\widetilde{\Phi}(0) \leq 0$, we obtain the reduction functional  $\widetilde{\Phi}$ attains its infimum in  $B_{R_2}$. Let $\beta_2 = \inf\limits_{u \in B_{R_2}}\widetilde{\Phi}(u)$, then Lemma \ref{lem:3.1} and Lemma \ref{lem:4.2} show that $\widetilde{\Phi}$ satisfies the $(PS)_{\beta_2}$ condition. Thus, there exists $u_2 \in  E_2$ such that $\widetilde{\Phi}' (u_2)= 0$ and $\widetilde{\Phi} (u_2)= \beta_2$.

Now we claim $u_1$ and $u_2$ are two different critical points in $E_2$. Observing  $0 \in B_{R_2}$, the following inequality
\begin{eqnarray}\label{eqa:6.2}
\beta_2=\inf\limits_{u \in B_{R_2}}\widetilde{\Phi}(u) \leq \widetilde{\Phi}(0) \leq 0 < \tau_2 \leq \inf\limits_{\|u\|_E=R_2}\widetilde{\Phi}(u)\leq \sup\limits_{u \in E_2}\widetilde{\Phi} (u)=\beta_1
\end{eqnarray}
shows $u_1 \neq u_2$.

If $\widetilde{\Phi}$ possesses another local maximum point which is different from $u_1$, then there exist at least three critical points of $\widetilde{\Phi}$.

Otherwise, if $u_1$ is the unique local maximum point of $\widetilde{\Phi}$. We shall prove that there exists a critical point of mountain-pass type which is different from $u_1$ and $u_2$.

Taking $u_0 \in E_2$ with $\|u_0\|_E = 1$, by Lemma \ref{lem:5.1}, there exists $\widetilde{R} > R_2$ such that $\widetilde{\Phi} (\widetilde{R} u_0) \leq 0$. Moreover, one of the fact holds: for any $s \in [0, 1]$,  either $s\widetilde{R} u_0 \neq u_1$ or $-s\widetilde{R} u_0 \neq u_1$.

\emph{Case 1}: if $s\widetilde{R} u_0 \neq u_1$ for any $s \in [0, 1]$, by Lemma \ref{lem:5.3} and \eqref{eqa:6.1}, we obtain
$$\max \{\widetilde{\Phi} (0), \widetilde{\Phi} (\widetilde{R} u_0)\} \leq 0 < \tau_2 \leq \inf\limits_{\|u\|_E= R_2} \widetilde{\Phi} (u).$$

Let
$$c^+ = \inf\limits_{g \in \Delta^+} \max\limits_{s \in [0, 1]} \widetilde{\Phi} (g(s)),$$
 where $\Delta^+ = \{g \in C([0, 1], E_2) : g(0)=0, g(1)=\widetilde{R} u_0\}$. In addition, we have  $\widetilde{\Phi}$ is $C^1$ function and satisfies $(PS)_{c^+}$ condition. By  mountain pass lemma \cite{{Chang.1986}}, we obtain  a critical point $u^+_3$ satisfying $\widetilde{\Phi} (u^+_3) = c^+$ and $\widetilde{\Phi}' (u^+_3) = 0$. Obviously, $c^+ \geq \tau_2 >0$.

It is easy to see that $s\widetilde{R} u_0 \in \Delta^+$ for any $s \in [0, 1]$. Therefore, the assumption $s\widetilde{R} u_0 \neq u_1$ implies
 $$c^+ \leq \max\limits_{s \in [0, 1]} \widetilde{\Phi} (s\widetilde{R} u_0) < \sup\limits_{u\in E_2}\widetilde{\Phi}(u)= \beta_1.$$
Thus, the following inequalities
$$\beta_2 = \inf\limits_{u \in B_{R_2}}\widetilde{\Phi}(u)\leq 0< \tau_2 \leq c^+  < \sup\limits_{u\in E_2}\widetilde{\Phi}(u) =\beta_1$$
shows $u_1\neq u_2\neq u^+_3$.

\emph{Case 2}: similarly, if $-s\widetilde{R} u_0 \neq u_1$ for any $s \in [0, 1]$, we have
$$c^- = \inf\limits_{g \in \Delta^-} \max\limits_{s \in [0, 1]} \widetilde{\Phi} (g(s))$$
is the critical value of $ \widetilde{\Phi}$, where
 $\Delta^- = \{g \in C([0, 1], E_2) : g(0)=0, g(1)=-\widetilde{R} u_0\}$. Therefore, there exists $u^-_3 \in E_2$ satisfying $ \widetilde{\Phi} (u^-_3) = c^-$ and $ \widetilde{\Phi}' (u^-_3) = 0$. Furthermore, we have $u_1\neq u_2\neq u^-_3$.

 Consequently, by Lemma \ref{lem:3.1}, we obtain  the  functional $\Phi$ has at least three critical points. The proof is completed.

\end{proof}

\vskip 5mm

{\bf Ethics statement.} This work did not involve any active collection of human data.

\vskip 5mm

{\bf Data accessibility statement.} This work does not have any experimental data.

\vskip 5mm

{\bf Competing interests statement.} We have no competing interests.

\vskip 5mm

{\bf Authors¡¯ contributions.}
HW and SJ contributed to the mathematical proof and writing the paper. All authors gave final approval for publication.

\vskip 5mm

{\bf Funding.} This work was supported by NSFC Grants (nos. 11322105 and 11671071).



%% References


\bibliographystyle{elsarticle-num}
\bibliography{<your-bib-database>}

\section*{References}

\begin{thebibliography}{00}
\bibitem {Amann.(1979)}
H. Amann, Saddle points and multiple solutions of differential equations, Math. Z. 169 (1979) 127--166.

\bibitem {A.(1995)}
D. Arcoya, D.G. Costa, Nontrivial solutions for a strongly resonant problem, Differential Integral Equations 8 (1995) 151--159.

\bibitem {B.(1983)}
H. Br\'{e}zis, Periodic solutions of nonlinear vibrating strings and duality principles, Bull. Amer. Math. Soc. (N.S.) 8 (1983) 409--426.

\bibitem {B.(1995)}
A.K. Ben-Naoum, J. Berkovits, On the existence of periodic solutions for semilinear wave equation on a ball in $\mathbb{R}^n$ with the space dimension $n$ odd, Nonlinear Anal. 24 (1995) 241--250.

\bibitem {B.(1978)}
H. Br\'{e}zis, L. Nirenberg, Forced vibrations for a nonlinear wave equation, Comm. Pure Appl. Math. 31 (1978) 1--30.

\bibitem {Barbu.(1996)}
V. Barbu, N.H. Pavel, Periodic solutions to one-dimensional wave equation with piece-wise constant coefficients, J. Differential Equations 132 (1996) 319--337.

\bibitem {Barbu.(1997)a}
V. Barbu, N.H. Pavel, Periodic solutions to nonlinear one dimensional wave equation with $x$-dependent coefficients, Trans. Amer. Math. Soc. 349 (1997) 2035--2048.

\bibitem {Barbu.(1997)b}
V. Barbu, N.H. Pavel, Determining the acoustic impedance in the 1-D wave equation via an optimal control problem, SIAM J. Control Optim. 35 (1997) 1544--1556.

\bibitem {Chang.(1981)}
K. Chang, Solutions of asymptotically linear operator equations via Morse theory, Comm. Pure Appl. Math. 34 (1981) 693--712.

\bibitem {Chang.1986}
K. Chang, Critical Point Theory and Its Applications, Shanghai Scientific and Technical Publishers, Shanghai, 1986 (in Chinese).

\bibitem{C.1982}
K. Chang, S. Wu, S. Li, Multiple periodic solutions for an asymptotically linear wave equation, Indiana Univ. Math. J. 31 (1982) 721--731.

\bibitem {Castro.(1979)}
A. Castro, A.C. Lazer, Critical point theory and the number of solutions of a nonlinear Dirichlet problem, Ann. Mat. Pura Appl. 120 (1979) 113--137.

\bibitem {Chen.(2014)}
J. Chen, Z. Zhang, Infinitely many periodic solutions for a semilinear wave equation in a ball in $\mathbb{R}^n$, J. Differential Equations 256 (2014) 1718--1734.

\bibitem {Chen.(2016)}
J. Chen, Z. Zhang, Existence of infinitely many periodic solutions for the radially symmetric wave equation with resonance, J. Differential Equations 260 (2016) 6017--6037.

\bibitem {Chen.(2017)}
J. Chen, Z. Zhang, Existence of multiple periodic solutions to asymptotically linear wave equations in a ball, Calc. Var. Partial Differential Equations 56 (2017) 58.

\bibitem {Craig.(1993)}
W. Craig, C.E. Wayne, Newton's method and periodic solutions of nonlinear wave equations, Comm. Pure Appl. Math. 46 (1993) 1409--1498.

\bibitem{I.2000}
M. Izydorek, Multiple solutions for an asymptotically linear wave equation, Differential Integral Equations 13 (2000) 289--310.

\bibitem {Ji.(2008)}
S. Ji, Time periodic solutions to a nonlinear wave equation with $x$-dependent coefficients, Calc. Var. Partial Differential Equations 32 (2008) 137--153.

\bibitem {Ji.(2009)}
S. Ji, Time-periodic solutions to a nonlinear wave equation with periodic or anti-periodic boundary conditions, Proc. R. Soc. Lond. Ser. A 465 (2009) 895--913.

\bibitem{Ji.(2016)}
S. Ji, Y. Gao,  W. Zhu, Existence and multiplicity of periodic solutions for Dirichlet-Neumann boundary value problem of a variable coefficient wave equation, Adv. Nonlinear Stud. 16 (2016) 765--773.

\bibitem {Ji.(2006)}
S. Ji, Y. Li, Periodic solutions to one-dimensional wave equation with $x$-dependent coefficients, J. Differential Equations 229 (2006) 466--493.

\bibitem {Ji.(2007)}
S. Ji, Y. Li, Time-periodic solutions to the one-dimensional wave equation with periodic or anti-periodic boundary conditions, Proc. Roy. Soc. Edinburgh Sect. A 137 (2007) 349--371.

\bibitem {Ji.(2011)}
S. Ji, Y. Li, Time periodic solutions to the one-dimensional nonlinear wave equation, Arch. Ration. Mech. Anal. 199 (2011) 435--451.

\bibitem {R.(1967)}
P.H. Rabinowitz, Periodic solutions of nonlinear hyperbolic partial differential equations, Comm. Pure Appl. Math. 20 (1967) 145--205.

\bibitem {R.(1978)}
P.H. Rabinowitz, Free vibrations for a semilinear wave equation, Comm. Pure Appl. Math. 31 (1978) 31--68.

\bibitem {R.(1984)}
P.H. Rabinowitz, Large amplitude time periodic solutions of a semilinear wave equation, Comm. Pure Appl. Math. 37 (1984) 189--206.

\bibitem{Ru04}
I.A. Rudakov, Periodic solutions of a nonlinear wave equation with nonconstant coefficients, Math. Notes 76 (2004) 395--406.

\bibitem{Ru17}
I.A. Rudakov, Periodic solutions of the quasilinear equation of forced vibrations of an inhomogeneous string, Math. Notes 101 (2017) 137--148.

\bibitem {Schechter.(1998)}
M. Schechter, Rotationally invariant periodic solutions of semilinear wave equations, Abstr. Appl. Anal. 3 (1998) 171--180.

\bibitem {T.(2006)}
M. Tanaka, Existence of multiple weak solutions for asymptotically linear wave equations, Nonlinear Anal. 65 (2006) 475--499.

\bibitem{W.2009}
P. Wang, Y. An, Resonance in nonlinear wave equations with $x$-dependent coefficients, Nonlinear Anal. 71 (2009) 1985--1994.

\bibitem {Y.(1980)}
K. Yosida, Functional Analysis, sixth ed., Springer-Verlag, Berlin, 1980.

\end{thebibliography}



\end{document}